\documentclass[11pt,double,reqno]{amsart}

\usepackage{tikz}
\usepackage{times}
\usepackage{graphicx}
\usepackage{amssymb}
\usepackage{epsfig}
\usepackage{amsmath}
\usepackage{amsthm}
\usepackage{color}

\newtheorem{theo}{Theorem}[section]

\newtheorem{lemm}[theo]{Lemma}
\newtheorem{rem}[theo]{Remark}

\numberwithin{equation}{section}

\newcommand{\al}{\alpha}
\newcommand{\be}{\beta}

\newcommand{\Ga}{\Gamma}
\newcommand{\la}{\lambda}

\newcommand{\om}{\omega}
\newcommand{\Om}{\Omega}

\newcommand{\si}{\sigma}

\newcommand{\ep}{\epsilon }
\newcommand{\te}{\theta}
\newcommand{\De}{\Delta}
\newcommand{\de}{\delta}

\newcommand{\pa}{\partial}

\newcommand{\R}{{\mathbb R}^n}

\newcommand{\ri}{\rightarrow}
\newcommand{\Rn}{{\mathbb R}^{n-1}}

\newcommand{\na}{\nabla}

\begin{document}
\baselineskip=18pt

\title[]{Global well-posedness of the half space problem of the Navier-Stokes equations in critical function spaces of limiting  case}

\
\author{Tongkeun Chang}
\address{Department of Mathematics, Yonsei University \\
Seoul, 136-701, South Korea}
\email{chang7357@yonsei.ac.kr}

\author{Bum Ja Jin}
\address{Department of Mathematics, Mokpo National University, Muan-gun 534-729,  South Korea }
\email{bumjajin@mokpo.ac.kr}

\thanks{Tongkeun Chang was  supported by NRF-2017R1D1A1B03033427 and Bum Ja Jin was supported by NRF-2016R1D1A1B03934133.
}

\begin{abstract}
In this  paper, we study the  initial-boundary value problem of the Navier-Stokes equations in  half-space. Let a solenoidal  initial velocity be given in  the function space $ \dot{B}_{p\infty,0}^{ -1 + n/p}({\mathbb R}^n_+)$   for  $ \frac{n}3< p < n$.  We prove the global in time existence of weak solution $u\in L^\infty(0,\infty; \dot B^{-1 +n/p}_{p\infty}({\mathbb R}^n_+))$, when the  given initial velocity   has small norm in function space $ \dot{B}_{p\infty,0}^{-1 + n/p} ({\mathbb R}^n_+)$, where $ \frac{n}3< p< n$.
    \\



\end{abstract}

\maketitle

\section{\bf Introduction}
\setcounter{equation}{0}

In this  paper, we study the initial-boundary value problem of Navier--Stokes equations
\begin{align}\label{maineq2}
\begin{array}{l}\vspace{2mm}
u_t - \De u + \na p =-{\rm div}(u\otimes u), \qquad {\rm div} \, u =0 \mbox{ in }
 \R_+\times (0, \infty),\\
 u(x,0) = u_0(x) \quad x \in \R_+, \qquad  u(x',0, t) = 0 \quad x' \in \Rn, \,\, t \in (0, \infty),
\end{array}
\end{align}
where
 $u=(u_1,\cdots, u_n)$ and $p$ are the unknown velocity and pressure, respectively, and
   $     u_0=(u_{01},\cdots, u_{0n})$ is the given initial data.

Because the  Navier--Stokes equations are invariant under the scaling
\begin{align*}
u_\la(x,t) = \la u(\la x, \la^2 t),\quad   p_\la (x,t) = \la^2 p(\la x, \la^2t), \quad
u_{0 \la} (x)= \la u_0 (\la x),
\end{align*}
it is important to
study \eqref{maineq2} in the so-called critical spaces, i.e., the function spaces
with norms invariant under the scaling $u(x,t) \ri \la u(\la x,\la^2 t)$. The homogeneous Besov space $\dot B^{-1 +\frac{n}p}_{pq} (\R_+)$ is one of critical spaces.

In this paper, we prove the existence of global time mild solution $u\in L^\infty(0,\infty;\dot B^{-1 +\frac{n}p}_{p\infty} (\R_+))$ of \eqref{maineq2} for the initial data  $ u_0 \in \dot B^{-1 +\frac{n}p}_{p \infty,0} (\R_+)$, $\frac{n}3 < p < n$. See Section \ref{notation} for definitions of function spaces.

There are a number of papers dealing with global well-posedness for \eqref{maineq2} in homogeneous Besov space $\dot B^{-1 +\frac{n}p}_{pq} (\R_+), \, 1 < p, q < \infty$ (see \cite{chang-jin2,CJ2,CJ3,CJ4,DZ,farwig-giga-hsu,FSV,KSe2}  and the references therein).

The limiting case $ 1 < p < \infty, \, q=\infty$ has been studied by M. Cannone, F.  Planchon, and  M. Schonbek \cite{CPS} for $u_0 \in L^3 ({\mathbb R}^3_+)$,
by  H. Amann  \cite{amann} for $u_0 \in b^{-1 +\frac{n}p}_{p,\infty} (\Om)$, $ p > \frac{n}3,p \neq  n$, where $\Om$ is a standard domain like  ${\mathbb R}^3, {\mathbb R}^3_+$, exterior or bounded domain  in ${\mathbb R}^3$ and   by  M. Ri, P. Zhang and Z. Zhang \cite{RZZ} for  $u_0 \in b^{0}_{n \infty} (\Om)$,
where $ \Om$ is  $\R$, $\R_+$ or bounded domain with smooth boundary, and  $b^s_{p\infty}(\Omega)$ denotes the completion of the generalized Sobolev space $H^s_p(\Omega)$ in $B^s_{p\infty}(\Omega)$.
In particular, in \cite{CPS}, the solution $ u \in L^\infty_{\frac12 -\frac3{2p}} (0, \infty, L^3({\mathbb R}^3_+))$ exists globally in time when $\|u_0\|_{\dot B^{-1 +\frac3{p}}_{p \infty}({\mathbb R}^3_+)}, p > 3$ ($L^3(\R_+) \subset \dot B^{-1 +\frac3p}_{p \infty}({\mathbb R}^3_+)$) is small enough. Recently, H.  Kozono and S. Shimizu \cite{KSe} showed the existence of solution $u \in L^\infty (0, \infty; \dot B^{-1 +\frac{n}p}_{p \infty} (\R)), \,\, p > n$  of \eqref{maineq2} with sufficiently small norm of $u_0$ in $\dot B^{-1 +\frac{n}p}_{p\infty, 0} (\R)$. See also   \cite{fernandes,FGH,giga1,kozono1,Ko-Ya,Ma,sol1} and the references therein for the initial   value problem of Navier-Stokes equations in the half space.

Our study of this paper is motivated by the results in \cite{amann,CPS,KSe,RZZ}.
The following text states our main results.
\begin{theo}
\label{thm-navier}
Let   $\frac{n}3<p<n$. Assume that  $u_0\in \dot {B}_{p\infty,0}^{-1 +\frac{n}p }(\R_+)$ with $ \mbox{\rm div}\, u_0=0$. Then,    there is $\epsilon_*>0$  so that   if $\|u_0\|_{ \dot {B}_{p_0\infty,0}^{-1 + \frac{n}{p_0}} (\R_+) }<\ep_*$ for some   $n<p_0<\infty$,  then
 \eqref{maineq2} has a solution
 $u\in L^\infty(0,\infty;\dot B^{-1 +\frac{n}p}_{p\infty} (\R_+))$ $\cap L^{\infty}_{\frac12 -\frac{n}{2p_0}}( 0, \infty;L^{p_0} (\R_+) )$.
\end{theo}
 \begin{rem}
 In Theorem \ref{thm-navier}
  observe that   $\dot {B}_{p\infty,0}^{-1 +\frac{n}p }(\R_+) \subset \dot {B}_{p_0\infty,0}^{-1 +\frac{n}{p_0}} (\R_+)$ holds (see Theorem 6.5.1 in \cite{BL}).
\end{rem}

For the proof of Theorem \ref{thm-navier}, it is necessary to study the following initial value problem of the Stokes equations in $\R_+\times (0,\infty)$:
\begin{align}\label{maineq-stokes}
\begin{array}{l}\vspace{2mm}
u_t - \De u + \na p =f, \qquad {\rm div} \, u =0 \mbox{ in }
 \R_+\times (0,\infty),\\
\hspace{30mm}u|_{t=0}= u_0, \qquad  u|_{x_n =0} = 0,
\end{array}
\end{align}
where $f=\mbox{div}{\mathcal F}$.

In \cite{GGS}, M. Giga, Y. Giga and H. Sohr showed that if $f \in L^q(0, T; \hat D(A^{-\al}_p))$  and $u_0=0$ then the solution $u$ of Stokes equations  \eqref{maineq-stokes} satisfies that for $0 < \al <1$,
\begin{align*}
\int_0^T \Big( \| (\frac{d}{dt})^{1-\al}  u(t) \|^q_{L^p (\Om)} + \| A^{1-\al}_p u(t)\|^q_{L^p (\Om)}  \Big)dt \leq c(p,q, \Om,\al) \int_0^T \|A_p^{-\al}f (t) \|^q_{L^p (\Om)} dt,
\end{align*}
where   $A_p $ is Stokes operator in $\Om$ for standard domain $\Om$ such as bounded domain, exterior domain or half space, and   $ \hat D(T)$ is the completion of $D(T)$ in the homogeneous norm $\|T\cdot\|$
. In particular,  if $f ={\rm div}\, {\mathcal F}$ with ${\mathcal F} \in L^q(0, T; L^p_\si (\Om))$ then
\begin{align*}
\int_0^T \Big(\| (\frac{d}{dt})^{\frac12}  u(t) \|^q_{L^p (\Om)} + \| \na u(t)\|^q_{L^p (\Om)} \Big) dt \leq c(p,q, \Om) \int_0^T \|{\mathcal F}(t) \|^q_{L^p (\Om)} dt.
\end{align*}

H. Koch and V. A. Solonnikov~\cite{KS} showed  the unique local in time existence of solution $u \in L^q(0,T; L^q({\mathbb R}^3_+))$  of \eqref{maineq-stokes}
when $f=\mbox{div}{\mathcal F}$, ${\mathcal F}\in L^q(0,T; L^q({\mathbb R}^3_+))$ and  $u_0=0$. See also  \cite{giga1,KS1,kozono1,sol1,Sol-2} and the references therein.

The following theorem states our result on the unique solvability of the Stokes equations \eqref{maineq-stokes}.
\begin{theo}\label{thm-stokes}
Let $1 < p_1 \leq p < \infty$  and $\al > 0$. Let  $u_0\in \dot {B}_{p\infty,0}^{-\al}(\R_+)$ with $ \mbox{\rm div}\, u_0=0$ and  $f=\mbox{div}\mathcal F$.
\begin{itemize}
\item[(1)]
Let $ -1 +\al < \frac{n}{p_1} -\frac{n}p < 1$ and    ${\mathcal F} \in L^{\infty}_{\frac12 \al +\frac12 -\frac{n}2(\frac1{p_1} -\frac1p)  } (0, \infty,  L^{p_1 }(\R_+)) $. There is a solution  $u$ of \eqref{maineq-stokes}   with
\begin{align*}
\|u\|_{L^\infty_{\frac12 \al } (0, \infty; L^p (\R_+))}  \leq c \big( \| u_0\|_{\dot B^{-\al}_{p\infty, 0} (\R_+)} +  \| {\mathcal F}\|_{L^{\infty}_{\frac12 \al +\frac12 -\frac{n}2(\frac1{p_1} -\frac1p)  } (0,\infty;  L^{p_1}(\R_+))} \big).
\end{align*}
\item[(2)]
 Let $ 0<  \al  < 2$   such that $ 1 + \frac{n-1}p <  \frac{n}{p_1} < 1 + \frac{n}p $ and     ${\mathcal F} \in L^{\infty}_{\frac12 -\frac{n}{2p_1} +\frac{n}{2p}} (0,\infty;  \dot B^{\al }_{p_1\infty}(\R_+)) $. There is a solution  $u$ of \eqref{maineq-stokes}   with
\begin{align*}
\| u\|_{L^\infty (0, \infty;  \dot B^\al_{p\infty}(\R_+))} \leq c \big( \| u_0\|_{ \dot {B}_{p\infty,0}^{\al}(\R_+)} + \| {\mathcal F}\|_{L^{\infty}_{\frac12 -\frac{n}{2p_1} +\frac{n}{2p}} (0,\infty;  \dot B^{\al }_{p_1\infty}(\R_+))} \big).
\end{align*}

\end{itemize}
\end{theo}


This paper is organized as follows.
In Section \ref{notation}, we introduce  the function spaces and definition of the weak solutions of Stokes equations and Navier-Stokes equations.  In Section \ref{preliminary},    various  estimates of operators related to a  Newtonian  kernel and Gaussian kernel are given.  
In Section \ref{theoremthm-stokes}, we  complete the proof of Theorem \ref{thm-stokes}. In Section \ref{nonlinear}, we give the proof of Theorem \ref{thm-navier} while applying the estimates in  Theorem \ref{thm-stokes} to the  approximate solutions.

\section{Notations, Function spaces and  Definitions of weak solutions}
\label{notation}
We denote by  $x'$ and $x=(x',x_n)$ the points of the spaces $\Rn$ and $\R$, respectively.
The multiple derivatives are denoted by $ D^{k}_x D^{m}_t = \frac{\pa^{|k|}}{\pa x^{k}} \frac{\pa^{m} }{\pa t}$ for multi-index
$ k$ and nonnegative integers $ m$.
Throughout this paper we denote by $c$ various generic constants. 

%

For $s\in {\mathbb R}$ and $1\leq p,q\leq \infty$,
 we denote   $\dot H^s_{p}(\R)$ and $\dot{B}^s_{pq}(\R)$  the generalized homogeneous Sobolev spaces(space of Bessel potentials) and the homogeneous Besov spaces in $\R$, respectively (see \cite{BL,Tr} for the definition of function spaces).  
 Denoted by $\dot H^s_p(\R_+)$ and $ \dot B^s_{pq}(\R_+)$ are  the restrictions of $\dot H^s_p (\R)$ and $  \dot B^s_{pq} (\R)$, respectively,    with norms
 \begin{align*}
\| f \|_{\dot H^s_p(\R_+) }  = \inf \{ \| F\|_{ \dot H^s_p(\R)}\, | \, F|_{\R_+} =f, \,\,  F \in  \dot H^s_p(\R) \},\\
\| f \|_{\dot B^s_{pq}(\R_+) }  = \inf \{ \| F\|_{ \dot B^s_{pq}(\R)}\, | \, F|_{\R_+} =f, \,\, F \in  \dot B^s_{pq}(\R) \}.
 \end{align*}
For    a non-negative integer $k$, $\dot H^k_p(\R_+)
 = \{ f \, | \,   \sum_{|l| =k} \| D^l f\|_{L^p (\R_+)} < \infty \}.$
In particular, $\dot H^0_p(\R_+) = L^p (\R_+)$. 


For $s\in {\mathbb R}$, we denote  $\dot{B}^s_{pq,0}(\R_+), 1\leq p,q\leq \infty$  by
 \begin{align*}
 \dot{B}^s_{pq,0}(\R_+ ) & = \{ f \in  \dot{B}^s_{pq}(\R_+) \, | \,  \tilde f \in \dot B^s_{pq} ({\mathbb R}^n) \},
 \end{align*}
where $ \tilde f$ is the zero extension of $f$ over $\R$. Note that   $\| \tilde f \|_{\dot B^s_{pq} (\R)} \leq c \| f\|_{\dot B^s_{pq,0} (\R_+)}$.

Note that for $s \geq 0$,   $\dot B^{-s}_{pq, 0} (\R_+), \, 1 < p, q \leq \infty$ is the dual space of  $\dot B^{s}_{p'q'} (\R_+)$, that is,  $\dot B^{-s}_{pq, 0} (\R_+) = (\dot B^s_{p' q'}(\R_+))^*$, where $\frac1p + \frac1{p'} =1$ and $ \frac1q +\frac1{q'} =1$.

For the Banach space $X$, we denote  by $L_\be^\infty (0, \infty ;X)$, $\be \geq 0$  the usual Bochner space with norm
\begin{align*}
\| f\|_{L^\infty_\be (0, \infty; X)} : = \sup_{0 < t < \infty} t^\be  \| f(t) \|_X.
\end{align*}
%

For  $1\leq q\leq \infty$ and   $0< \theta<1$, we    denote by   $(X,Y)_{\theta,q}$ and $[X, Y]_\te$  the real interpolation  and complex interpolation of the Banach spaces $X$ and $Y$. In particular, for $ 0< \te < 1 $, $  \al, \al_1, \al_2 \in {\mathbb R} $ and  $1 \leq  p, q, r \leq \infty$,
\begin{align}
\label{interpolation0}
[\dot H^{\al_1}_{p}(\R_+), \dot H^{\al_2}_{p} (\R_+)]_{\te } = \dot H^\al_{p}(\R_+),\\
\label{interpolation1-3}
 [\dot B^{\al}_{p_1 q}(\R_+), \dot B^{\al_2}_{pq} (\R_+)]_{\te} = \dot B^\al_{pr}(\R_+),\\
\label{interpolation1}
(\dot H^{\al_1}_{p}(\R_+), \dot H^{\al_2}_{p} (\R_+))_{\te, r} = \dot B^\al_{pr}(\R_+),
\end{align}
when $\al = \te \al_1 + (1 -\te) \al_2$. See Theorem 6.4.5, Theorem 5.1.2 and Theorem 5.6.2 in \cite{BL}.

From now, we denote $\dot B^\be_{p\infty}: = \dot B^\be_{p\infty}(\R_+)$, $\dot H^\be_{p}: = \dot H^\be_{p}(\R_+)$,   $L^\infty_\al B^\be_{p\infty}: = L^\infty_\al (0, \infty; \dot B^\be_{p\infty}(\R_+))$ and $L^\infty_\al L^{p}: = L^\infty_\al(0, \infty; \dot L^{p}(\R_+))$.

\section{\bf Preliminary Estimates.}

\label{preliminary}
\setcounter{equation}{0}

\subsection{Newtonian potential}

The fundamental solution of the Laplace equation in $\R$ is  denoted by
\[
   N(x) = \left\{\begin{array}{ll}
 \vspace{2mm}
  \frac{1}{\om_n (2-n)|x|^{n-2}}&\mbox{if }n\geq 3,\\
 \frac{1}{2\pi}\ln |x|&\mbox{if }n=2,\end{array}\right.
 \]
$\omega_n$ is the surface area of the unit sphere in $\R$.
\begin{lemm}\label{lemma0929-22}
For $\al \geq  0$ and $1 < p < \infty$.
\begin{align*}
\| \nabla_x^2 \int_{{\mathbb R}^n_+} N(\cdot-y)    f(y) dy \|_{ \dot H^\al_{p} } \leq c \| f\|_{ \dot H^\al_{p}  }.
\end{align*}
\end{lemm}
The proof of Lemma \ref{lemma0929-22} is given  in Appendix \ref{appendixa}.

\subsection{Gaussian kernel}

The fundamental solution of the heat equation in $\R$ is denoted by
\[
 \Gamma(x,t)=\left\{\begin{array}{ll} \vspace{2mm}
 \frac{1}{ (2\pi t)^{\frac{n}{2}}}e^{-\frac{|x|^2}{4t}}&\mbox{if }t>0,\\
 0& \mbox{if }t\leq 0.
 \end{array}\right.
 \]

Let  $\Ga^* (x-y, t): = \Ga(x' -y',x_n +y_n,t)$. For $u_0 \in B^{\al}_{p \infty, 0}, \, 1\leq p \leq \infty, \, \al  \in {\mathbb R}$, we define  $\Ga_t* u_{0j}$ and $\Ga_t^*  * u_{0j}$ by
\begin{align*}
&\Ga_t* u_{0j}(x) :=  \int_{{\mathbb R}^n_+}\Ga(x-y,t) u_{0j}(y)  dy,\quad \Ga_t^* * u_{0j}(x)
:=  \int_{{\mathbb R}^n_+}\Ga^*(x-y, t) u_{0j}(y ) dy \quad \al \geq 0,\\
& \Ga_t* u_{0j}(x) := < u_0, \Ga(x-\cdot, t)>,\quad \Ga_t^* * u_{0j}(x) :=  < u_0, \Ga^*(x-\cdot, t)> \quad \al < 0,
\end{align*}
where $< \cdot, \cdot>$ is dual paring between $\dot B^{\al}_{p\infty} $ and $\dot B^{-\al}_{p'1,0}$.

\begin{lemm}\label{lemm0315}
\begin{itemize}
\item[(1)]
Let $ 1 \leq   p \leq  \infty$ and    $\al > 0$. Let $ u_0 \in \dot B^{ -\al }_{p \infty,0} $ with ${\rm div} \, u_0 =0$. Then,
\begin{align}
 \| \Ga_t * u_0 \|_{L^\infty_{\frac12 \al} L^p }, \,\, \|\Ga_t^* * u_0 \|_{L^\infty_{\frac12 \al } L^p } \leq c \| u_0\|_{\dot B^{ -\al  }_{p \infty} }.
\end{align}

\item[(2)]
Let $ 1 \leq    p \leq \infty $ and    $\al > 0$. Let $ u_0 \in \dot B^{ \al }_{p \infty,0} $ with ${\rm div} \, u_0 =0$. Then,
\begin{align}
 \| \Ga_t * u_0 \|_{L^\infty \dot B^\al_{p \infty} }, \,\, \|\Ga_t^* * u_0 \|_{L^\infty\dot B^\al_{p \infty} } \leq c \| u_0\|_{\dot B^{ \al }_{p \infty} }.
\end{align}
\end{itemize}
\end{lemm}
\begin{proof}
The proof of Lemma \ref{lemm0315} is given in Appendix \ref{appendix0718}.
\end{proof}

We define $\Gamma *f  $  and $\Gamma^* *f$ by
\begin{align*}
\Gamma *f (x,t):=\int^t_0\int_{\R_+}\Gamma(x-y,t-s)f(y,s)dyds,\\
\Gamma^* *f(x,t):=\int^t_0\int_{\R_+}\Gamma^*(x-y,t-s)f(y,s)dyds.
\end{align*}

\begin{lemm}\label{0929-1}
\begin{itemize}
\item[(1)]
Let $1 < p_1 \leq p < \infty$  and $\al > 0$ satisfying $ -1 +\al < \frac{n}{p_1} -\frac{n}p < 1$.  Let $f=\mbox{div}\mathcal F$ with    ${\mathcal F} \in L^{\infty}_{\frac12 \al +\frac12 -\frac{n}2(\frac1{p_1} -\frac1p)  } (0, \infty,  L^{p_1 }(\R_+)) $. Then,
\begin{align*}
\|  \Ga* {\mathbb P}f\|_{L^\infty_{\frac12 \al } (0, \infty; L^p (\R_+))}, \,\, \|  \Ga^* * {\mathbb P}f\|_{L^\infty_{\frac12 \al } (0, \infty; L^p (\R_+))} \leq c\| {\mathcal F}\|_{L^{\infty}_{\frac12 \al +\frac12 -\frac{n}2(\frac1{p_1} -\frac1p)  } (0,\infty;  L^{p_1}(\R_+))}.
\end{align*}
\item[(2)] Let $ 0 <   \al  < 2$ and $1 \leq p_1 \leq p$ such that $ 1 + \frac{n-1}p <  \frac{n}{p_1} < 1 + \frac{n}p $.  Let $f=\mbox{div}\mathcal F$ with    ${\mathcal F} \in L^{\infty}_{\frac12 -\frac{n}{2p_1} +\frac{n}{2p}} (0,\infty;  \dot B^{\al }_{p_1\infty}(\R_+)) $. Then,
\begin{align*}
\|  \Ga* {\mathbb P}f\|_{L^\infty (0, \infty;  \dot B^\al_{p\infty}(\R_+))}, \,\, \|  \Ga^* * {\mathbb P}f\|_{L^\infty(0, \infty;  \dot B^\al_{p\infty}(\R_+))} \leq c\| {\mathcal F}\|_{L^{\infty}_{\frac12 -\frac{n}{2p_1} +\frac{n}{2p}} (0,\infty;  \dot B^{\al }_{p_1\infty}(\R_+))}.
\end{align*}

\end{itemize}
\end{lemm}
\begin{proof}
The proof  of Lemma \ref{0929-1} is given in Appendix \ref{appendix0292-1}.
\end{proof}

%
%

\subsection{ H$\ddot{\rm o}$lder type inequality}
The following H$\ddot{\rm o}$lder type inequality  is a well-known result (see Lemma 2.2 in \cite{chae}).
\begin{lemm}\label{0510prop}
Let $0 < \be  $ and $1 \leq p, q \leq \infty$. Then, for   $\frac1{r_i} + \frac1{s_i} = \frac1p$, $i=1,2$,
\begin{align*}
\| f_1f_2\|_{\dot B^{\be}_{pq}  }  \leq
        c \big( \| f_1\|_{ \dot  B^{\be }_{s_1 q}  } \| f_2\|_{  L^{r_1} }   + \| f_1\|_{ L^{s_2 }} \|f_2\|_{ \dot  B^{\be }_{r_2 q}}  \big).
\end{align*}
\end{lemm}

\subsection{Helmholtz projection ${\mathbb P}$} \label{projection}
Let ${\mathbb P}$ be a  Helmholtz projection operator in $\R_+$. Let $f=\mbox{div}{\mathcal F}, \ {\mathcal F}=(F_{kl})_{k,l=1}^n$, $F_{kl}=F_{lk}$, with $F_{mk}|_{x_n=0}=0$.
 Then ${\mathbb P}f$ can be rewritten as
${\mathbb P}f=\mbox{div}{\mathcal F}'$,  ${\mathcal F'}=(F_{km}')_{k,m=1}^n$, where
\begin{align*}
F'_{nm}&=F_{nm}-\delta_{nm}F_{nn}, \ m=1,\cdots, n,\\
F'_{\beta \gamma}&=F_{\beta \gamma}-\delta_{\beta \gamma}F_{nn}+D_{x_\gamma}
\Big(\sum_{q=1}^n \int_{\R_+}D_{y_q}N^+(x,y)F_{\beta q}(y) dy \\
&\quad+\int_{\R_+} \Big( D_{y_n}N^+(x,y)F_{n\beta}(y)-D_{y_\beta}N^+(x,y)F_{nn}(y)\Big) dy   \Big) \quad  \beta,\gamma\neq n,\\
F'_{\beta n}&=-\sum_{\gamma=1}^{n-1}D_{x_\gamma}\int_{\R_+}D_{x_n}N^+(x,y)F_{\beta \gamma}(y) dy+D_{x_\beta}\int_{R_+}D_{x_n}N^+(x,y) F_{nn}(y)dy\\
&\quad-2F_{\beta n}(x)-2\sum_{\gamma=1}^{n-1}D_{x_\gamma}\int_{\R_+}D_{x_\gamma}N^{-}(x,y)F_{\beta n}(y) dy \quad \beta\neq n.
\end{align*}
(See Section 3 of  \cite{CM}  for the details).
Here $N^+(x,y):=N(x-y)+N(x-y^*)$ and $N^{-
}(x,y)=N(x-y)-N(x-y^*)$. From Lemma \ref{lemma0929-22} and real interpolation, we have
\begin{align}
\notag \|{\mathcal F}'\|_{\dot{H}^\al_p} & \leq c\|{\mathcal F}\|_{\dot{H}^\al_p} \quad 0 \leq \al , \,\, 1<p<\infty,\\
\label{0718-1} \|{\mathcal F}'\|_{\dot{B}^\al_{p\infty}}& \leq c\|{\mathcal F}\|_{\dot{B}^\al_{p\infty}} \quad  0 < \al, \,\,  1<p<\infty.
\end{align}

\section{Proof of  Theorem \ref{thm-stokes}}
\label{theoremthm-stokes}
\setcounter{equation}{0}

First, we decompose the Stokes equations \eqref{maineq-stokes} as the following two equations:
\begin{align}
\label{stokes.zero}
\notag v_t - \De v+\nabla \pi  =0, \qquad {\rm div} \, v =0 \qquad \mbox{ in
}\,\,\R_+ \times (0,\infty),\\
v|_{t =0} = u_0 \quad \mbox{ and }\quad v|_{x_n =0} =0,
\end{align}
and
\begin{align}\label{maineq-stokesh=0}
\begin{array}{l}\vspace{2mm}
V_t - \De V + \na \Pi =\mbox{div}{\mathcal F}, \qquad \mbox{div } V =0 \mbox{ in }
 \R_+\times (0,\infty),\\
\hspace{30mm}V|_{t=0}= 0, \qquad  V|_{x_n =0} = 0.
\end{array}
\end{align}

Let $u = V + v$ and $p =\pi+\Pi$. Then, $(u,p)$ is a solution of  \eqref{maineq-stokes}.

\subsection{Estimate of $(v,\pi)$}

The solution $(v,\pi)$ of \eqref{stokes.zero} is represented by (see  \cite{So})
\begin{equation}\label{expression-v-zero}
v_i (x,t) = \int_{{\mathbb R}^n_+} G_{ij}(x,y, t)
u_{0j}(y) dy,
\end{equation}
\begin{equation}\label{expression-p-zero}
\pi(x,t) = \int_{{\mathbb R}^n_+} P(x,y, t)
\cdot  u_0 (y)dy,
\end{equation}
where $G$ and $P$ are defined by
\begin{align}\label{formulas-v}
\notag G_{ij} &= \de_{ij} (\Ga(x-y, t) - \Ga(x-y^*,t ))\\
& \qquad  + 4(1 -\de_{jn})
\frac{\pa}{\pa x_j} \int_0^{x_n} \int_{{\mathbb R}^{n-1}}
            \frac{\pa N(x-z)}{\pa x_i} \Ga(z -y^* , t) dz,
\end{align}
\begin{align}\label{formulas-p}
\notag P_j(x,y,t) & =4 (1 - \de_{jn}) \frac{\pa }{\pa x_j}\Big[ \int_{{\mathbb
R}^{n-1}} \frac{\pa N(x' - z', x_n)}{\pa x_n} \Ga(z' -y', y_n,t) dz'\\
& \quad +\int_{{\mathbb R}^{n-1}} N(x' -z',x_n) \frac{\pa \Ga(z'-y', y_n,
t)}{\pa y_n} dz'\Big].
\end{align}

From proofs of Lemma 3.1 and Lemma 3.3   in \cite{CK}, we get
\begin{align}\label{v-initial}
\notag \| v(t)\|_{\dot H^{k}_{p}} &\leq c \big(  \|\Ga_t * u_0\|_{\dot H^{k}_{p }} +  \|\Ga_t^** u_0\|_{\dot H^{k}_{p }} \big) \quad 1< p< \infty, \,\, k \in {\mathbb N} \cup \{ 0\},\\
\| v(t)\|_{\dot B^{\al}_{p\infty}} & \leq c \big(  \|\Ga_t * u_0\|_{\dot B^{\al}_{p\infty }} +  \|\Ga_t^** u_0\|_{\dot B^{\al}_{p\infty }} \big) \quad 1 < p < \infty, \quad \al > 0.
\end{align}

\subsection{Estimate of $(V,\Pi)$}

Let  ${\mathbb P}f$ be  the  Helmholtz projection of $f$.  Note that ${\rm div}\, {\mathbb P} f =0$ and  $({\mathbb P} f)_n|_{x_n =0} =0$.  We define  $(V, \Pi_0)$  by
\begin{equation}\label{expression-V}
V_i (x,t) =\int_0^t \int_{{\mathbb R}^n_+} G_{ij}(x,y, t-\tau)
({\mathbb P} f)_j(y,\tau) dyd\tau,
\end{equation}
\begin{equation}\label{expression-p}
\Pi_0(x,t) =\int_0^t \int_{{\mathbb R}^n_+} P(x,y, t-\tau)
\cdot  ({\mathbb P} f) (y,\tau)dyd\tau,
\end{equation}
where $G$ and $P$ are defined by
\eqref{formulas-v}and \eqref{formulas-p}.
Then  $(V, \Pi_0)$ satisfies
\begin{align*}
\begin{array}{l}\vspace{2mm}
V_t - \De V + \na \Pi_0 ={\mathbb P} f, \qquad \mbox{div } V =0, \mbox{ in }
 \R_+\times (0,\infty),\\
\hspace{30mm}V|_{t=0}= 0, \qquad  V|_{x_n =0} = 0.
\end{array}
\end{align*}
(See \cite{So}.) Let $\Pi = \Pi_0 + {\mathbb Q}f$. Then, $(V, \Pi)$ is the solution of \eqref{maineq-stokesh=0}.

Let $1<p<\infty$, $1 \leq q \leq \infty$, and $0 \leq \al \leq 2$.  In Section 3 in  \cite{CK}, the authors showed that
 $V $ defined by
\eqref{expression-V}   has the following estimates (using real interpolations);
\begin{align}\label{0417-2}
\notag \|   V(t) \|_{  \dot H^{k}_{p} } &
   \leq c \big( \|   \Gamma* {\mathbb P} f(t)\|_{    \dot H^{k}_{p}} +
  \|   \Gamma^* * {\mathbb P} f(t) \|_{   \dot H^{k}_{p}} \big) \quad 1< p< \infty, \,\, k \in {\mathbb N} \cup \{ 0\},\\
\|   V(t) \|_{  \dot B^{\al}_{p\infty} } &
   \leq c \big( \|   \Gamma* {\mathbb P} f(t)\|_{    \dot B^{\al}_{p\infty}} +
  \|   \Gamma^* * {\mathbb P} f(t) \|_{   \dot B^{\al}_{p\infty}} \big) \quad 1< p< \infty, \,\, \al > 0.
\end{align}

%
%
%

\subsection{Estimate of $(u,p)$}

Note that $(u, p)$ defined by $u = V + v$ and $p = \pi + \Pi_0 + {\mathbb Q} \,  {\rm div} \, {\mathcal F}$ is the solution of \eqref{maineq-stokes}. From \eqref{v-initial},  \eqref{0417-2}, Lemma \ref{lemm0315} and Lemma \ref{0929-1}, we obtain Theorem \ref{thm-stokes}.

\section{Nonlinear problem}

\label{nonlinear}
\setcounter{equation}{0}

In this section, we would like to give  proof of Theorem \ref{thm-navier}. For this purpose, we  construct approximate velocities and then
derive uniform convergence in  $L^\infty  \dot B^\al_{p\infty}$.

\subsection{Approximating solutions}

Let $(u^1,p^1)$ be the solution of the Stokes equations
\begin{align}
\begin{array}{l}\vspace{2mm}
u^1_t - \De u^1 + \na p^1 =0, \qquad \mbox{div } u^1 =0, \mbox{ in }
 \R_+\times (0,\infty),\\
\hspace{30mm}u^1|_{t=0}= u_0, \qquad  u^1|_{x_n =0} = 0.
\end{array}
\end{align}
Let $m\geq 1$.
After obtaining $(u^1,p^1),\cdots, (u^m,p^m)$ construct $(u^{m+1}, p^{m+1})$ which satisfies the following equations
\begin{align}
\label{maineq5}
\begin{array}{l}\vspace{2mm}
u^{m+1}_t - \De u^{m+1} + \na p^{m+1} =f^m, \qquad \mbox{div } u^{m+1} =0, \mbox{ in }
 \R_+\times (0,\infty),\\
\hspace{30mm}u^{m+1}|_{t=0}= u_0, \qquad  u^{m+1}|_{x_n =0} = 0,
\end{array}
\end{align}
where $f^m=-\mbox{div}(u^m\otimes u^m)$.

\subsection{Uniform boundedness in $L^\infty_{\frac12 -\frac{n}{2p_0}}L^{p_0}$,  $n < p_0$}
\label{uniform1}

From  (1) of  Theorem \ref{thm-stokes}, we have
\begin{align}
\label{uc1-1}
\| u^{1}\|_{L^{\infty}_{\frac12 -\frac{n}{2p_0} }   L^{p_0} }
 \leq c_0  \|u_0\|_{  \dot B^{-1 +n/p_0}_{p_0 \infty,0}}:=N_0.
\end{align}
From  (1) of Theorem \ref{thm-stokes}, taking $p_1 =\frac{p_0}2$, we have
\begin{align}\label{1011-1-2}
\notag  \| u^{m+1}\|_{L^{\infty}_{\frac12 -\frac{n}{2p_0}}  L^{p_0} } & \leq  c \big( \| u_0\|_{\dot B^{-1 +n/p_0}_{p_0 \infty,0} }  +    \| u^{m} \otimes u^{m} \|_{L^{\infty}_{1 -\frac{n}{p_0}} L^{p_0/2} } \big)\\
  & \leq  c_1 \big( \| u_0\|_{\dot B^{ -1 + n/p_0}_{p_0 \infty,0} }  +    \| u^{m}  \|^2_{L^\infty_{\frac12 -\frac{n}{2p_0}} L^{p_0}} \big).
\end{align}

Under the hypothesis  $\|u^m\|_{L^{\infty}_{\frac12 -\frac{n}{2p_0}}  L^{p_0} }\leq M_0$,  \eqref{1011-1-2} leads to the estimate
\[
\|u^{m+1}\|_{L^{\infty}_{\frac12 -\frac{n}{2p_0}}  L^{p_0} }\leq c_1 \big( N_0+  M_0^2 \big).
\]
Choose $M_0$ and $N_0$  so small that
\begin{align}\label{0508-2}
 M_0  \le \frac{1}{2c_1}\mbox{ and }N_0<\frac{M_0}{2c_1}.
\end{align}
By the mathematical induction argument, we conclude
\begin{align}\label{0509-1}
\|u^{m}\|_{ L^{\infty}_{\frac12 -\frac{n}{2p_0}} L^{p_0} }\leq M_0 \,  \mbox{ for all } \, m=1,2\cdots.
\end{align}

\subsection{Uniform boundedness in $L^\infty\dot B^{-1 +n/p}_{p\infty}$}
\label{uniform2}

From (2) of   Theorem \ref{thm-stokes}, we have
\begin{align}
\label{0509-1-1}
\| u^{1}\|_{L^\infty  \dot B^{-1 + n/p}_{p\infty} }
& \leq c_2 \|u_0\|_{  \dot B^{-1+n/p}_{p\infty,0}}:=N.
\end{align}

We take $1 \leq p_1< p$ satisfying $1 + \frac{n-1}p < \frac{n}{p_1} < 1 + \frac{n}p$.
From (2) of   Theorem \ref{thm-stokes} to obtain
\begin{align}\label{1011-1-2-1}
 \| u^{m+1}\|_{L^\infty  \dot{B}^{-1 + n/p}_{p\infty} } & \leq  c \big( \| u_0\|_{\dot B^{-1+ n/p}_{p\infty,0} }  +    \| u^{m} \otimes u^{m} \|_{L^\infty_{\frac12 -\frac{n}{2p_1} +\frac{n}{2p}} \dot{B}^{-1 +n/p}_{p_1\infty} } \big).
\end{align}
Let $\frac1{p_1} = \frac1{p_0} +\frac1{p}$ such that $1 - \frac1p < \frac{n}{p_0} <1$. By Lemma \ref{0510prop},   we have
\begin{align}\label{0510-3-1}
\notag\|(u^m\otimes u^m)\|_{L^{\infty}_{\frac12 -\frac{n}{2p_1} +\frac{n}{2p}}  \dot B^{-1+ n/p}_{p_1\infty} }
& \leq c \|u^m \|_{L^\infty  \dot B^{-1+ n/p}_{p \infty} } \|u^m  \|_{L^\infty_{\frac12 -\frac{n}{2p_0}} L^{p_0} }\\
& \leq c \|u^m \|_{L^\infty \dot B^{ -1+ n/p }_{p\infty} } \|u^m  \|_{L^{ \infty}_{\frac12 -\frac{n}{2p_0}} L^{p_0} }.
\end{align}
From \eqref{1011-1-2-1}-\eqref{0510-3-1}, we have
\begin{align}\label{eq0412-3}
 \| u^{m+1}\|_{L^\infty \dot B^{-1+ n/p}_{p\infty}}
  \leq   c_1 \big( N
 +  \|u^m  \|_{L^{ \infty}_{\frac12 -\frac{n}{2p_0}} L^{p_0} } \|u^m\|_{ L^\infty\dot B^{-1+ n/p}_{p\infty}} \big).
\end{align}

Under the hypothesis
$\|u^m\|_{L^\infty \dot B^{-1+ n/p}_{p\infty}}\leq M$,
\eqref{eq0412-3} leads to the estimate
\[
\|u^{m+1}\|_{L^\infty \dot B^{-1+ n/p}_{p\infty}}\leq c_1 \big( N+  M_0M \big).
\]
Choose that  $M_0$ is  small  and $M$ is  large  so that that
\begin{equation}
\label{0508-2-2} M_0 \leq  \frac1{2c_1}, \qquad  2c_1 N\leq M.
\end{equation}
By the mathematical induction argument, we conclude
\begin{equation}
\label{0509-1-1}
\|u^{m}\|_{L^\infty \dot B^{-1+ n/p}_{p\infty}}\leq M \mbox{ for all }m=1,2\cdots.
\end{equation}

\subsection{Uniform convergence}

Let $U^m=u^{m+1}-u^m$ and $P^m=p^{m+1}-p^m$.
Then, $(U^m,P^m)$ satisfy the equations
\[
\begin{array}{l}\vspace{2mm}
U^m_t - \De U^m + \na P^m =-\mbox{div}(u^m\otimes U^{m-1}+U^{m-1}\otimes u^{m-1}), \qquad \mbox{div } U^{m} =0, \mbox{ in }
 \R_+\times (0,\infty),\\
\hspace{30mm}U^{m}|_{t=0}= 0, \qquad  U^{m}|_{x_n =0} =0.
\end{array}
\]

Recall  the uniform estimates \eqref{0509-1} and \eqref{0509-1-1} for the approximate solutions. From  Theorem \ref{thm-stokes} and  Lemma \ref{0510prop} we have
\begin{align}\label{0425-3}
\notag\| U^m\|_{L^{\infty}_{\frac12 -\frac{n}{2p_0}} L^{p_0} }
& \leq c \big( \| u^{m-1}\|_{L^{\infty}_{\frac12 -\frac{n}{2p_0}} L^{p_0} }  + \| u^{m}\|_{L^{\infty}_{\frac12 -\frac{n}{2p_0}}L^{p_0} }  \big) \| U^{m-1}\|_{L^{\infty}_{\frac12 -\frac{n}{2p_0}}L^{p_0} }\\
 & \leq c_5 M_0 \| U^{m-1}\|_{L^{\infty}_{\frac12 -\frac{n}{2p_0}} L^{p_0} },
\end{align}
and
\begin{align}\label{0425-2}
\notag \|U^m\|_{L^\infty \dot B^{-1+ n/p}_{p\infty}}
&\leq c\|u^m\otimes U^{m-1}+U^{m-1}\otimes u^{m-1}\|_{L^{\infty}\dot B^{-1+ n/p}_{p_1\infty} }\\
\notag &\leq c \big(\|u^m\|_{L^\infty \dot B^{-1+ n/p}_{p\infty}} + \|u^{m-1}\|_{L^\infty \dot B^{-1+ n/p}_{p\infty}} \big) \|U^{m-1}\|_{L^{\infty}_{\frac12 -\frac{n}{2p_0}} L^{p_0}}\\
    \notag &\qquad + c\big( \|u^m\|_{L^{ \infty}_{\frac12 -\frac{n}{2p_0}} L^{p_0}}  + \|u^{m-1}\|_{L^{\infty}_{\frac12 -\frac{n}{2p_0}}L^{p_0}} \big)\|U^{m-1}\|_{L^\infty\dot B^{-1+ n/p}_{p\infty}}  \big)\\
     &\leq c_6 M \|U^{m-1}\|_{L^{\infty}_{\frac12 -\frac{n}{2p_0}}  L^{p_0} }    +   c_6M_0\|U^{m-1}\|_{L^\infty \dot B^{-1+ n/p}_{p\infty}}.
\end{align}
We take the constant $c_6$ greater than $c_5$, that is,
\begin{align}
c_6>c_5.
\end{align}

From \eqref{0425-3}, if $c_5M_0<1$, then $\sum_{m=1}^\infty\| U^m\|_{ L^{\infty}_{\frac12 -\frac{n}{2p_0}}  L^{p_0} }$ converges, that is,
$$\sum_{m=1}^\infty U^m\mbox{ converges in }L^\infty_{ \frac12 -\frac{n}{2p_0}}  L^{p_0} .$$

Take $A>0$ satisfying $A(c_6-c_5)M_0\geq c_6 M$. Then from \eqref{0425-3} and \eqref{0425-2} it holds that
\begin{align*}
 \|U^m\|_{L^\infty \dot B^{-1+ n/p}_{p\infty}}+A\| U^m\|_{L^{\infty}_{\frac12 -\frac{n}{2p_0}} L^{p_0}}\\
 \leq c_6M_0(\|U^{m-1}\|_{L^\infty\dot B^{-1+ n/p}_{p\infty}}+A\| U^{m-1}\|_{L^{\infty}_{\frac12 -\frac{n}{2p_0}} L^{p_0}})
 \end{align*}
 Again if $c_6M_0<1$, then $\sum_{m=1}^\infty(\|U^m\|_{L^\infty  \dot B^{-1+ n/p}_{p\infty}}+A\| U^m\|_{L^{\infty}_{\frac12 -\frac{n}{2p_0}} L^{p_0}})$ converges. This implies that
 $\sum_{m=1}^\infty\|U^m\|_{L^\infty \dot B^{-1+ n/p}_{p\infty}}$ converges, that is, $$\sum_{m=1}^\infty U^m\mbox{ converges in }L^\infty \dot B^{-1+ n/p}_{p\infty}.$$

Therefore, if $M_0$ satisfies the condition \eqref{0508-2-2} with the additional conditions
\begin{align}
M_0< \frac{1}{c_6},
\end{align}
then  $u^m=u^1+\sum_{k=1}^mU^{k}$ converges to $u^1+\sum_{k=1}^\infty U^{k}$ in  $L^\infty \dot B^{-1+ n/p}_{p\infty} \cap L^{\infty}_{\frac12 -\frac{n}{2p_0}}L^{p_0} $.
Set $u:=u^1+\sum_{k=1}^\infty U^{k}.$

\subsection{Existence}
In this section, we will show that $u$ satisfies weak formulation of the Navier-Stokes equations \eqref{maineq2}, that is, $u$ is a weak solution of the Navier-Stokes equations \eqref{maineq2} with appropriate distribution $p$.

Let $u$ be the same one constructed in  the previous Section. Because $u_m \ri u$ in $L^{\infty}_{\frac12 -\frac{n}{2p_0}}L^{p_0} $  by \eqref{0509-1}, we have
\[
\|u\|_{L^{\infty}_{\frac12 -\frac{n}{2p_0}} L^{p_0} }, \quad \|u^m\|_{L^{\infty}_{\frac12 -\frac{n}{2p_0}} L^{p_0} }\leq  M_0.\]

Let $\Phi\in C^\infty_{0}({\R_+} \times [0,T))$ with $\mbox{div }\Phi=0$ for some $T > 0$. Observe that
\begin{align*}
-\int^\infty_0\int_{\R_+} u^{m+1}\cdot \Delta\Phi dxdt&=\int^\infty_0\int_{\R_+}u^{m+1}\cdot \Phi_t+(u^m\otimes u^m): \nabla \Phi dxdt  +\int_{\R_+} u_0 \cdot \Phi(x,0)dx.
\end{align*}
Now, send $m$ to the  infinity, then,  $u^m\rightarrow u$ in $L^{\infty}_{\frac12 -\frac{n}{2p_0}}   L^{p_0} $.   Then, we have
\begin{align*}
\int^T_0\int_{\R_+} \big( u^{m+1} - u \big)\cdot \Delta\Phi dxdt & \leq  \int^T_0 \|  u^{m+1} - u \|_{L^{p_0} } \| \Delta\Phi\|_{L^{p'_0}}dt\\
& \leq \|  u^{m+1} - u \|_{L^\infty_{\frac12  -\frac{n}{2p_0}}  L^{p_0} }   \int^T_0 t^{-\frac12  +\frac{n}{2p_0} }  \| \Delta\Phi\|_{L^{p'_0}}dt\\
& \ri 0 \quad \mbox{as         } m \ri \infty.
\end{align*}
Similarly, we have
\begin{align*}
\int^T_0\int_{\R_+} \big( u^{m+1} - u \big)\cdot \Phi_t dxdt &  \ri 0 \quad \mbox{as         } m \ri \infty.
\end{align*}
Because $p_0 > n$, we have
\begin{align*}
\int^\infty_0\int_{\R_+}\big(u^m\otimes ( u^m - u)\big): \nabla \Phi dxdt & \leq \|u^m\otimes(  u^{m+1} - u )\|_{L^{\infty}_{1  -\frac{n}{p_0}} L^{p_0/2}}   \int^T_0 t^{-1  +\frac{n}{p_0} }  \| \Delta\Phi\|_{L^{(p_0/2)'}}dt\\
& \leq \|u^m -u \|_{L^{\infty}_{\frac12   -\frac{n}{2p_0}}L^{p_0} } \|u^m \|_{L^{\infty}_{\frac12   -\frac{n}{2p_0}}L^{p_0}}  \int^T_0 t^{-1 +\frac{n}{p_0} }  \| \Delta\Phi\|_{L^{(p_0/2)'}}dt\\
& \ri 0 \quad \mbox{as         } m \ri \infty.
\end{align*}
Hence, we have the identity
\begin{align*}
-\int^\infty_0\int_{\R_+}u\cdot \Delta \Phi dxdt&=\int^\infty_0\int_{\R_+}u\cdot \Phi_t+(u\otimes u): \nabla\Phi dxdt +\int_{\R_+} u_0 \cdot \Phi(x,0)dx.
\end{align*}
Therefore,  $u$ is a weak solution of the Navier-Stokes equations \eqref{maineq2}.  This completes the proof of the existence part of Theorem \ref{thm-navier}.

\subsection{Uniqueness in space $L^{\infty}_{\frac12 -\frac{n}{2p_0}} L^{p_0} $}
Let 
  $ u_1\in  L^{\infty}_{\frac12 -\frac{n}{2p_0}}   L^{p_0} $ be  another weak solution of the Navier-Stokes equations \eqref{maineq2} with pressure $p_1$. Then,
 $(u-u_1,p-p_1)$ satisfies the equations
\begin{align*}
(u-u_1)_t - \De (u-u_1) + \na (p-p_1)& =-\mbox{div}(u\otimes (u-u_1)+(u-u_1)\otimes u_1)\mbox{ in }
 \R_+\times (0,\infty), \\
 {\rm div} \, (u-u_1)& =0,
 \mbox{ in }\R_+\times (0,\infty),\\
 (u-u_1)|_{t=0}= 0, &\quad (u-u_1)|_{x_n =0} =0.
\end{align*}
Applying  the estimate of Theorem \ref{thm-stokes} to the above Stokes equations,  we have
\begin{align*}
\| u-u_1\|_{L^{\infty}_{\frac12 -\frac{n}{2p_0}} L^{p_0}}
 \leq c \|u\otimes (u-u_1)+(u-u_1)\otimes u_1 \|_{L^{\infty}_{1 -\frac{n}{p_0}} L^{p_0}}\\
 \leq c_5 ( \|u\|_{L^{\infty}_{\frac12 -\frac{n}{2p_0}} L^{p_0}}+\|u_1\|_{L^{\infty}_{\frac12 -\frac{n}{2p_0}}L^{p_0}})\| u-u_1\|_{L^{\infty}_{\frac12 -\frac{n}{2p_0}} L^{p_0}}.
 \end{align*}
 Taking  $M_0$ to satisfy $2c_5 M_0 <1$, we have
\begin{align*}
\| u-u_1\|_{L^{\infty}_{\frac12 -\frac{n}{2p_0}} L^{p_0}}
 < \| u-u_1\|_{L^{\infty}_{\frac12 -\frac{n}{2p_0}} L^{p_0}}.
 \end{align*}
This implies that $u \equiv u_1$ in $\R_+ \times (0, \infty)$ and so we complete the proof of the uniqueness part of Theorem \ref{thm-navier}.

\appendix
\setcounter{equation}{0}

\section{Proof of Lemma \ref{lemma0929-22}}
\label{appendixa}
The following lemma is well known trace theorem  (see Theorem 6.6.1  in \cite{BL}).
\begin{lemm}\label{trace}
Let $1< p < \infty$. If   $ f \in \dot H^{\al}_{p}$  for $\al > \frac1p$, then $f |_{x_n =0} \in \dot B^{\al -\frac1p}_{pp} (\Rn)$ with
\begin{align*}
 \| f |_{x_n =0} \| _{\dot B_{pp}^{\al -1/p}(\Rn)} \leq c \| f\|_{\dot H^\al_{p} }.
\end{align*}
\end{lemm}

We define $Nf$ by
\begin{align}\label{N}
N f(x) = \int_{\Rn} N(x'-y',x_n)f(y')dy'.
\end{align}

Observe that $D_{x_n}Nf$ is the Poisson operator of the  Laplace equation in $\R_+$ and $D_{x_i}Nf=D_{x_n}NR_i'f$  for $i\neq n$, where   $R'=(R_1,\cdots, R_{n-1})$ is the $n-1$ dimensional Riesz operator.
The Poisson operator is bounded from $\dot{B}^{\al-\frac{1}{p}}_{pp}(\Rn)$ to $\dot H^\al_{p}(\R_+), \al \geq 0$ and $R'$ is bounded from $\dot B^s_{pp}(\Rn)$ to $\dot B^s_{pp}(\Rn)$, $s\in {\mathbb R}$ (see \cite{St}).
Hence  the following estimates hold.
\begin{lemm}
\label{poisson1}
Let  $1<p<\infty$. Then
\begin{align}\label{Poisson}
\| \nabla_x Nf\|_{\dot H^\al_{p}}\leq c\|f\|_{\dot B^{\al-1/p}_{pp}(\Rn)}\quad \al \geq 0.
 \end{align}
\end{lemm}

According to the Calder\'{o}n-Zygmund inequality
\begin{align}\label{cal-zyg}
\|  \int_{{\mathbb R}^n} \nabla_x^2N(\cdot-y)    f(y) dy \|_{ L^p} \leq c \| f\|_{ L^p}\quad \mbox{ for } \quad 1 < p < \infty.
\end{align}

We will show that for $k \geq 0$.
\begin{align}\label{0501-3}
\| D_x^k \nabla_x^2\int_{{\mathbb R}^n_+} N(x-y)    f(y) dy(x) \|_{L^p } \leq c \| D_x^kf\|_{L^p }.
\end{align}
Then, by the property of complex interpolation (see \eqref{interpolation0} and  \eqref{interpolation1-3}), we get  Lemma \ref{lemma0929-22}.

Note that
\begin{align*}
D_{x_i} \int_{{\mathbb R}^n_+} N(x-y)   f(y) dy  & =   \int_{{\mathbb R}^n_+} N(x-y)   D_{y_i} f(y) dy,\ i\neq n,\\
D_{x_n}  \int_{{\mathbb R}^n_+} N(x-y)    f(y) dy & =   \int_{{\mathbb R}^n_+} N(x-y)   D_{y_n} f(y) dy -  \int_{\Rn} N(x'-y', x_n) f(y', 0) dy'.
\end{align*}
By  \eqref{cal-zyg},  \eqref{Poisson} and Lemma \ref{trace},  we have
\begin{align*}
\notag \| \nabla_{x}^3 \int_{{\mathbb R}^n_+} N(x-y)    f(y) dy\|_{L^p }
& \leq c \big( \| \nabla_x f\|_{L^p } + \| f(\cdot, 0)\|_{\dot B^{1 -\frac1p}_{pp} (\Rn)} \big)\\
& \leq c  \| \nabla_x f\|_{L^p }.
\end{align*}
By the successive argument,  \eqref{0501-3} can be obtained for any multiple integer $k\geq 0$. $\Box$

\section{Proof of Lemma \ref{lemm0315}}
\label{appendix0718}
Fix a Schwartz function $\phi\in {\mathcal S} ({\mathbb
R}^{n})$ satisfying $\hat{\phi}(\xi) > 0$ on $\frac12 < |\xi|
  < 2$, $\hat{\phi}(\xi)=0$ elsewhere, and
$\sum_{j=-\infty}^{\infty} \hat{\phi}(2^{-j}\xi) =1$
for $\xi \neq 0$. Let
\begin{align*}
\widehat{\phi_j}(\xi) &:= \widehat{\phi}(2^{-j} \xi ), \qquad (j =
0, \pm 1, \pm 2 , \cdots),
\end{align*}
where $\hat f = {\mathcal F}(f)$ is a Fourier transform of $f$. Let $\Phi = \phi_{-1}  +
\phi_0 +\phi_1$  and $\Phi_j(\xi) = \Phi(2^{-j} \xi)$ such that
$supp \, \Phi_j \subset \{ 2^{-j -2} < |\xi| < 2^{-j +2} \}$ and
$\Phi \equiv 1$ in  $2^{j -1} < |\xi|   < 2^{j +1}$.
\begin{lemm}\label{multiplier2}
Let  $\rho_{tj}(\xi) =   \Phi_j ( \xi) e^{-t|\xi|^2}$ for each integer
$j$.  Then  $\rho_{tj}( \xi)$ are  $L^\infty(\R)$-multipliers with the
finite norm $M(t,j)$. Moreover for $t
> 0$
\begin{align}\label{multiplier2_2}
M(t,j) & \leq c e^{-\frac14 t2^{2 j}}\sum_{0 \leq i \leq n} t^i
2^{2 ij} \leq c  e^{-\frac18 t2^{2 j}}.
\end{align}
\end{lemm}
See Lemma 13 in \cite{CK2}.

Let $\tilde u_0(x) = u_0(x) $   be  a zero extension over $\R$ such that $\|\tilde u_0 \|_{\dot B^{\al }_{p \infty} (\R)} \leq c
\|u_0 \|_{\dot B^{\al }_{p \infty} }$. Let
\begin{align*}
v(x,t) = \int_{\R} \Ga (x-y,t) \tilde u_0(y) dy.
\end{align*}
Then, we have  $\| v\|_{L^\infty \dot B^{\al }_{p \infty} } \leq c \| v\|_{L^\infty (0, \infty; \dot B^{\al }_{p \infty} (\R))}$.

Using the dyadic partition of unity $ \sum_{j=-\infty}^{\infty} \hat
\phi(2^{-j} \xi) =1$ for $\xi \neq 0$, we can write
\begin{align*}
{\mathcal F} ( v * \phi_j ) (\xi,t) =
  \hat\phi (2^{-j} \xi) e^{-t|\xi |^2}
\widehat{ \tilde u_0}(\xi).
\end{align*}

For $t>0$ we have
\begin{align}\label{u_1}
\notag \|  v * \phi_j(t)\|_{L^p (\R)}  & = \Big( \int_{{\mathbb R}^n}
\left|{\mathcal F}^{-1} \Big(  e^{-t|\xi|^2}
\hat \phi_j(\xi) \;\widehat{ \tilde u_0}(\xi) \Big)(x)\right|^p dx \Big)^\frac1p\\
& = \Big( \int_{{\mathbb R}^n}
\left|{\mathcal F}^{-1} \Big(  \hat \Phi_j(\xi) e^{-t|\xi|^2}
\hat \phi_j(\xi) \;\widehat{ \tilde u_0}(\xi) \Big)(x)\right|^p dx \Big)^\frac1p.
\end{align}

By Lemma \ref{multiplier2},  we have
\begin{align*}
 t^{\frac12 \al} \| v(t) \|_{L^p (\R)} & \leq  t^{\frac12 \al} \sum_{-\infty< j< \infty} \| v(t) * \phi_j\|_{L^p (\R)}\\
 & \leq   t^{\frac12 \al} \sum_{-\infty< j < \infty}  M(t,j) \| \tilde  u_0 * \phi_j\|_{L^p(\R)}\\
& \leq c t^{\frac12 \al} \sum_{-\infty< j < \infty} 2^{j\al} 2^{-t2^{2j}} 2^{-j\al}  \| \tilde  u_0 *\phi_j\|_{L^{p}(\R)}\\
& \leq c   t^{\frac12 \al} \sum_{-\infty< j < \infty} 2^{j\al} 2^{-t2^{2j}} \| \tilde  u_0 \|_{\dot B^{-\al}_{p \infty} (\R)}\\
& \leq c   \| \tilde  u_0 \|_{\dot B^{-\al}_{p \infty} (\R)}
\end{align*}
and
\begin{align*}
 2^{\al j} \| v(t) * \phi_j\|_{L^p (\R)}
 & \leq   2^{\al j} M(t,j) \| \tilde  u_0 * \phi_j\|_{L^p(\R)}\\
& \leq c 2^{\al j} 2^{-t2^{2j}}  \| \tilde  u_0 *\phi_j\|_{L^{p}(\R)}\\
& \leq c   \| \tilde  u_0 \|_{\dot B^{\al}_{p \infty} (\R)}.
\end{align*}
We complete the proof of Lemma \ref{lemm0315}.

\section{Proof of Lemma \ref{0929-1}}

\label{appendix0292-1}

Because the proofs will be done  in the same way, we  prove only the case of $ \Ga^* * {\mathbb P}f$.

A crucial step in the proof of Lemma \ref{0929-1} is the following lemma, which is probably
known to experts, however, we could not find it in the literature and thus, we provide
its proof.

\begin{lemm}\label{lemma0128}
Let $X_i$ and $Y_i$, $i = 1,2$ be Banach spaces and  $0< t $ be fixed. Let
$T : L^1(0, t; X_i) \ri Y_i, i = 1,2$ be   linear operators such that
\begin{align*}
\| Tf \|_{Y_i} \leq M_i\int_0^t (t-s)^{-\be_i} \| f(s) \|_{X_i} ds, \quad i = 1,2 \quad \forall f \in L^1(0, t; X_i).
\end{align*}
Then,  for $0 <\te< 1$ and $1 \leq q \leq \infty$,
\begin{align*}
\| Tf \|_{(Y_1, Y_2)_{\te,q}} \leq M_1^\te M_2^{1 -\te}\int_0^t (t-s)^{-\be} \| f(s) \|_{(X_1, X_2)_{\te,q}} ds,
\end{align*}
where $\be = \be_1 \te + \be_2 (1 - \te)$.
\end{lemm}
\begin{proof}
See Lemma C.1 in \cite{CJ4}.

\end{proof}

\begin{lemm}\label{apendixlemma0718-2}
Let
\begin{align*}
w(x,t) = \int_0^t \int_{\Rn}
D_{x} \Ga(x'-y', x_n, t-\tau)  f (y',\tau) dy'd\tau.
\end{align*}
Then, for $\be > 0$,
\begin{align}\label{0413-3}
 \|w(t)\|_{L^p} \leq c \int_0^t (t -\tau)^{\frac1{2p} -1-\frac{\be}2} \| f(\tau)\|_{\dot B^{-\be}_p(\Rn)} d\tau.
\end{align}
\end{lemm}
\begin{proof}
Let $D_x = D_{x_n}$. Then,
\begin{align}\label{0413-1}
\notag \| w(t) \|_{L^p } &\leq \int_0^t \big(\int_0^\infty \frac{x_n^p}{(t-\tau)^{\frac{3p}2}} e^{-\frac{x_n^2}{t-\tau}} dx_n \big)^\frac1p
        \| \Ga'_{t-\tau} *'f(\tau)\|_{L^p(\Rn)} d\tau\\
    & =c  \int_0^t (t -\tau)^{\frac1{2p} -1}
        \| \Ga'_{t-\tau} *' f(\tau)\|_{L^p(\Rn)} d\tau.
\end{align}
Here, $\Ga'_t $ is a Gaussian kernel in $\Rn$ and $*'$ is convolution in $\Rn$. Then, we have
\begin{align}\label{0413-2}
 \notag \| \Ga'_{t-s} *'  f(\tau)\|_{L^p(\Rn)} & = \|\sum_{-\infty < k < \infty} \Phi'_k *' \Ga'_{t-s} *' \phi'_k *' f(\tau)  \|_{L^p(\Rn)}\\
\notag & \leq \sum_{-\infty < k < \infty} \|\Phi'_k *' \Ga'_{t-s}\|_{M_p(\Rn)} \| \phi'_k *' f(\tau)  \|_{L^p(\Rn)}\\
\notag & \leq c \sum_{-\infty < k < \infty} e^{-(t-s)2^{2k} } \| \phi'_k *' f(\tau)  \|_{L^p(\Rn)}\\
\notag  & \leq c\big( \sum_{-\infty < k < \infty} e^{-(t-s)\frac{p}{p-1}2^{2k} }2^{\be \frac{p}{p-1} k} \big)^{\frac{p -1}p}  \big( \sum_{-\infty < k < \infty} 2^{-\be pk}  \| \phi'_k *' f(\tau)  \|^p_{L^p(\Rn)} \big)^\frac1p\\
  & \leq c (t-s)^{-\frac{\be}2} \| f(\tau)\|_{\dot B^{-\be}_p(\Rn)}.
\end{align}
From \eqref{0413-1} and \eqref{0413-2}, we obtain \eqref{0413-3}.

Let $D_x = D_{x'}$. Then,
\begin{align}\label{0413-1-1}
\notag \| w(t) \|_{L^p } &\leq \int_0^t \big(\int_0^\infty \frac{1}{(t-\tau)^{\frac{p}2}} e^{-\frac{x_n^2}{t-\tau}} dx_n \big)^\frac1p
        \|D_{x'} \Ga'_{t-\tau} *'f(\tau)\|_{L^p(\Rn)} d\tau\\
    & =c  \int_0^t (t -\tau)^{\frac1{2p} -\frac12}
        \| D_{x'}\Ga'_{t-\tau} *' f(\tau)\|_{L^p(\Rn)} d\tau.
\end{align}
Here, $\Ga'_t $ is a Gaussian kernel in $\Rn$ and $*'$ is convolution in $\Rn$. Then, we have
\begin{align}\label{0413-2-2}
 \notag \|D_{x_n} \Ga'_{t-s} *'  f(\tau)\|_{L^p(\Rn)} & = \|\sum_{-\infty < k < \infty} D_{x'}\Phi'_k *' \Ga'_{t-s} *' \phi'_k *' f(\tau)  \|_{L^p(\Rn)}\\
\notag & \leq \sum_{-\infty < k < \infty} \|D_{x'} \Phi'_k *' \Ga'_{t-s}\|_{M_p(\Rn)} \| \phi'_k *' f(\tau)  \|_{L^p(\Rn)}\\
\notag & \leq c \sum_{-\infty < k < \infty}2^k e^{-(t-s)2^{2k} } \| \phi'_k *' f(\tau)  \|_{L^p(\Rn)}\\
\notag  & \leq c\big( \sum_{-\infty < k < \infty} 2^{\frac{p}{p-1} k} e^{-(t-s)\frac{p}{p-1}2^{2k} }2^{\be \frac{p}{p-1} k} \big)^{\frac{p -1}p}  \big( \sum_{-\infty < k < \infty} 2^{-\be pk}  \| \phi'_k *' f(\tau)  \|^p_{L^p(\Rn)} \big)^\frac1p\\
  & \leq c (t-s)^{-\frac12 -\frac{\be}2} \| f(\tau)\|_{\dot B^{-\be}_p(\Rn)}.
\end{align}
From \eqref{0413-1-1} and \eqref{0413-2-2}, we obtain \eqref{0413-3}.

\end{proof}

{\bf Proof of Lemma \ref{0929-1}}
Recalling Helmholtz decomposition of $f = {\rm div F}$ with $F|_{x_n =0} =0$ in Section \ref{projection}, we have
\begin{align}\label{0421-1-1}
\notag \Ga^* * ({\mathbb P}f )_j (x,t) &= \int_0^t \int_{{\mathbb R}^n_+}
  \Ga(x-y^*, t-\tau) \cdot  {\rm div} \, F'_{j}(y,\tau) dyd\tau\\
 &= -\int_0^t \int_{{\mathbb R}^n_+}
 \nabla_{y} \Ga(x-y^*, t-\tau) \cdot  F'_{j}(y,\tau) dyd\tau.
\end{align}

Using \eqref{0421-1-1},  Young's inequality and \eqref{0718-1},  for  $p_1 \leq p$, we have
\begin{align}\label{0508-1}
\notag \|\Ga^* * ({\mathbb P}\, f)(t)\|_{L^{p} }
 &\leq  c \int_0^t (t-s)^{-\frac12 -\frac{n}2 (\frac1{p_1} -\frac1{p})} \sum_{j=1}^{j =n} \|F'_j(s)\|_{L^{p_1}}  ds \\
& \leq  c \int_0^t (t-s)^{-\frac12 -\frac{n}2 (\frac1{p_1} -\frac1{p})} \| {\mathcal F}(s)\|_{L^{p_1}} ds.
\end{align}
Similarly, we get
\begin{align}\label{0508-1-1}
\notag \|\na \Ga^* * ({\mathbb P}\, f)(t)\|_{L^{p} }
&\leq  c \int_0^t (t-s)^{-\frac12 -\frac{n}2 (\frac1{p_1} -\frac1{p})} \sum_{j=1}^{n}\| \na F'_j(s)\|_{L^{p_1}} ds \\
&\leq  c \int_0^t (t-s)^{-\frac12 -\frac{n}2 (\frac1{p_1} -\frac1{p})} \| {\mathcal F}(s)\|_{\dot H^1_{p_1}} ds.
\end{align}

From  \eqref{0421-1-1},  we have
\begin{align}\label{0421-1}
\notag \na^2_x\Ga^* * ({\mathbb P}f )_j (x,t) &= -\int_0^t \int_{{\mathbb R}^n_+}
 \nabla_{y} \Ga(x-y^*, t-\tau) \cdot  \na^2_y F'_{j}(y,\tau) dyd\tau \\
& \quad + \int_0^t \int_{\Rn} D_{x} \Ga (x' -y', x_n, t-s) D_{y_n} F_j(y', 0, s) dy'd\tau.
\end{align}
Hence from Lemma \ref{apendixlemma0718-2}, for $\be > 0$, we have
\begin{align*}
\| D_x^2 \Ga^* * ({\mathbb P}\, f) (t) \|_{L^p } &\leq \int_0^t   (t-\tau)^{-\frac12 -\frac{n}{2p_1} +\frac{n}{2p}} \| D^2_y F'_{j}(\tau)\|_{L^{p_1}(\R_+)}  d\tau \\
&\quad + \int_0^t (t -\tau)^{\frac1{2p} -1-\frac{\be}2} \| D_y F'_{j}(\tau)\|_{\dot B^{-\be}_p(\Rn)}  d\tau\\
&: = I_1(t) + I_2(t).
\end{align*}
From \eqref{0718-1},  we have
\begin{align*}
I_1(t) \leq c\int_0^t (t-\tau)^{-\frac12 -\frac{n}{2p_1} +\frac{n}{2p}}  \| F(\tau)\|_{\dot H^2_{p_1}} d\tau.
\end{align*}
Because  $1 +\frac{n-1}p < \frac{n}{p_1} < 1 +\frac{n}p$, taking $\be  = -1   +\frac{n}{p_1} - \frac{n-1}p > 0$, we have
\begin{align*}
 I_2(t)  & \leq c \int_0^t (t -\tau)^{\frac1{2p} -1-\frac{\be}2} \| D_y F'_{j}(\tau)\|_{\dot B^{1 -\frac{1}{p_1} }_{p_1}(\Rn)} d\tau\\
& \leq c \int_0^t (t -\tau)^{\frac1{2p} -1-\frac{\be}2} \| D_y F'_{j}(\tau)\|_{\dot H^{1  }_{p_1}(\R_+)} d\tau\\
&\leq  c  \int_0^t (t -\tau)^{-\frac12 -\frac{n}2(\frac1{p_1} -\frac1p)}
        \|    F(\tau)\|_{\dot H^2_{p_1}(\R_+)} d\tau.
\end{align*}
Hence, we obtain
\begin{align}\label{0718-3}
\| D_x^2 \Ga^* * ({\mathbb P}\, f) (t) \|_{L^p } \leq  c\int_0^t (t-\tau)^{-\frac12 -\frac{n}{2p_1} +\frac{n}{2p}}  \| F(\tau)\|_{\dot H^2_{p_1}} d\tau.
\end{align}

Then, from \eqref{0508-1},   we get
\begin{align*}
 \| \Ga^* * ({\mathbb P}\, f)(t)\|_{L^p  }
 & \leq  c \int_0^t (t-s)^{-\frac12 -\frac{n}2 (\frac1{p_1} -\frac1{p})} \| {\mathcal F}(s)\|_{L^{p_1}} ds\\
  &  \leq  c  \sup_{0 < s <t} \big( s^{\frac12 \al +\frac12 -\frac{n}2 (\frac1{p_1} -\frac1{p}) } \| {\mathcal F}(s)\|_{L^{p_1}} \big) \int_0^t (t-s)^{-\frac12 -\frac{n}2 (\frac1{p_1} -\frac1{p})}
      s^{- \frac12 \al -\frac12 +\frac{n}2 (\frac1{p_1} -\frac1{p}) } ds\\
   &  =  c  t^{-\frac12  \al  } \sup_{0 < s <t} \big(  s^{\frac12 \al +\frac12 -\frac{n}2 (\frac1{p_1} -\frac1{p}) }   \| {\mathcal F}(s)\|_{L^{p_1 }} \big).
\end{align*}
Hence, we complete the proof of  (1) of  Lemma \ref{0929-1}.

From  \eqref{0508-1}, \eqref{0508-1-1},    Lemma \ref{lemma0128} and \eqref{interpolation1}, for $0 < \al < 1$ and $\frac{n}{p_1} < 1 + \frac{n}p$,  we get
\begin{align}\label{0723-2}
\notag \| \Ga^* * ({\mathbb P}\, f)(t)\|_{\dot B^{\al}_{p\infty}  }
 & \leq  c \int_0^t (t-s)^{-\frac12 -\frac{n}2 (\frac1{p_1} -\frac1{p})} \| {\mathcal F}(s)\|_{\dot B^{\al}_{p_1\infty}} ds\\
\notag  &  \leq  c  \sup_{0 < s <t} \big( s^{\frac12 -\frac{n}{2p_1} +\frac{n}{2p}} \| {\mathcal F}(s)\|_{\dot B^{\al}_{p_1\infty}} \big) \int_0^t (t-s)^{-\frac12 -\frac{n}2 (\frac1{p_1} -\frac1{p})}  s ^{-\frac12 +\frac{n}{2p_1} -\frac{n}{2p}} ds\\
   &  =  c  \sup_{0 < s <t} \big( s^{\frac12 -\frac{n}{2p_1} +\frac{n}{2p}}  \| {\mathcal F}(s)\|_{\dot B^{\al}_{p_1\infty}} \big).
\end{align}
Hence, we obtain (2) of  Lemma \ref{0929-1} for $0 < \al <1$.

From \eqref{0508-1-1}, \eqref{0718-3}, Lemma \ref{lemma0128}  and \eqref{interpolation1}, for $1 < \al < 2$ and  $1< p_1< p$ with $1 +\frac{n-1}p < \frac{n}{p_1} < 1 +\frac{n}p$,  we have
\begin{align}\label{0719-1}
\| \Ga^* * ({\mathbb P}\, f) (t)\|_{\dot B^\al_{p \infty} } & \leq \int_0^t (t -s)^{-\frac12 -\frac{n}2 (\frac1{p_1} - \frac1p)} \| F(s)\|_{\dot B^{\al}_{p_1 \infty}} ds
\end{align}
As the same estimate of \eqref{0723-2},  we complete the proof of (2) of  Lemma \ref{0929-1}.

\end{document}